\documentclass[]{amsart}

\usepackage[english]{babel}
\usepackage[utf8]{inputenc}
\usepackage[T1]{fontenc}
\usepackage{microtype}
\usepackage{paralist}
\usepackage{amsmath, amsfonts, amssymb, amsthm}
\usepackage{color}
\usepackage[normalem]{ulem}
\usepackage{tikz}

\newcommand{\CC}{\mathbb{C}}

\newcommand{\NN}{\mathbb{N}}

\newcommand{\ZZ}{\mathbb{Z}}

\newcommand{\Hc}{\mathcal{H}}

\newcommand{\Oc}{\mathcal{O}}

\newcommand{\gfr}{\mathfrak{g}}
\newcommand{\set}[1]{\left\{ #1 \right\}}
\newcommand{\setb}[1]{\left( #1 \right)}
\newcommand{\abs}[1]{\left| #1 \right|}

\newcommand{\bino}[2]{\begin{pmatrix} #1 \\ #2 \end{pmatrix}}

\newtheorem{mymasterthm}{notForUse}
\theoremstyle{definition}

\theoremstyle{plain}
\newtheorem{mylemma}[mymasterthm]{Lemma}
\newtheorem{mythm}[mymasterthm]{Theorem}

\allowdisplaybreaks

\title[Asymptotics for Pillai's problem with polynomials]{Asymptotics for Pillai's problem with polynomials}
\makeatletter
\@namedef{subjclassname@2020}{\textup{2020} Mathematics Subject Classification}
\makeatother
\subjclass[2020]{11B37, 11D45}
\keywords{Polynomial power sums, Pillai problem}

\author[S. Heintze]{Sebastian Heintze}
\address{Sebastian Heintze\newline
	\indent Graz University of Technology\newline
	\indent Institute of Analysis and Number Theory\newline
	\indent Steyrergasse 30/II \newline
	\indent A-8010 Graz, Austria}
\email{heintze@math.tugraz.at}

\thanks{Supported by Austrian Science Fund (FWF) under project I4406}

\begin{document}
	
	\maketitle
	
	
	\begin{abstract}
		Let $ a_1(x)p_1(x)^n + \cdots + a_k(x)p_k(x)^n $ as well as $ b_1(x)q_1(x)^m + \cdots + b_l(x) q_l(x)^m $ be two polynomial power sums where the complex polynomials $ p_i(x) $ and $ q_j(x) $ are all non-constant.
		Then in the present paper we will give an asymptotic for the number of pairs $ (n,m) \in \NN^2 $ such that the degree of the sum of these two power sums is between $ 0 $ and $ d $ when $ d $ goes to infinity.
	\end{abstract}
	
	\section{Introduction}
	
	About hundred years ago Pillai considered in \cite{pillai-1931} exponential Diophantine equations of the shape
	\begin{equation*}
		a^n - b^m = f
	\end{equation*}
	for given positive integers $ a,b,f $ to be solved in integers $ n,m \geq 2 $.
	He proved for fixed integers $ a > 1 $ and $ b > 1 $ the asymptotic
	\begin{equation*}
		\# \set{(n,m) \in \NN^2 : 0 < a^n - b^m \leq x} \sim \frac{(\log x)^2}{2 \log a \log b}
	\end{equation*}
	as $ x \to \infty $.
	
	In \cite{kreso-tichy-} Kreso and Tichy considered the analogous situation for polynomials.
	Namely, they prove that for fixed non-constant coprime complex polynomials $ p(x) $ and $ q(x) $ we have the asymptotic
	\begin{equation*}
		\# \set{(n,m) \in \NN^2 : 0 \leq \deg \setb{p(x)^n - q(x)^m} \leq d} \sim \frac{d^2}{\deg p \deg q}
	\end{equation*}
	as $ d \to \infty $.
	It remains an open question in \cite{kreso-tichy-} to generalize the asymptotic result to more general polynmomial power sums.
	The purpose of the present paper is to provide such a generalization of the asymptotic result to polynomial power sums having dominant roots, see Theorem \ref{thm:genasymp} below.
	
	\section{Notation and results}
	
	Let us denote by $ F $ a function field in one variable over $ \CC $ and by $ \gfr $ the genus of $ F $.
	We will work with valuations and give here for the readers convenience a short wrap-up of this notion that can e.g.\ also be found in \cite{fuchs-heintze-p8}:
	For $ c \in \CC $ and $ f(x) \in \CC(x) $, where $ \CC(x) $ is the rational function field over $ \CC $, we denote by $ \nu_c(f) $ the unique integer such that $ f(x) = (x-c)^{\nu_c(f)} p(x) / q(x) $ with $ p(x),q(x) \in \CC[x] $ such that $ p(c)q(c) \neq 0 $. Further we write $ \nu_{\infty}(f) = \deg q - \deg p $ if $ f(x) = p(x) / q(x) $.
	These functions $ \nu : \CC(x) \rightarrow \ZZ $ are up to equivalence all valuations in $ \CC(x) $.
	If $ \nu_c(f) > 0 $, then $ c $ is called a zero of $ f $, and if $ \nu_c(f) < 0 $, then $ c $ is called a pole of $ f $, where $ c \in \CC \cup \set{\infty} $.
	For a finite extension $ F $ of $ \CC(x) $ each valuation in $ \CC(x) $ can be extended to no more than $ [F : \CC(x)] $ valuations in $ F $. This again gives up to equivalence all valuations in $ F $.
	Both, in $ \CC(x) $ as well as in $ F $ the sum-formula
	\begin{equation*}
		\sum_{\nu} \nu(f) = 0
	\end{equation*}
	holds, where the sum is taken over all valuations (up to equivalence) in the considered function field.
	For a finite set $ S $ of valuations on $ F $, we denote by $ \Oc_S^* $ the set of $ S $-units in $ F $, i.e. the set
	\begin{equation*}
		\Oc_S^* = \set{f \in F^* : \nu(f) = 0 \text{ for all } \nu \notin S}.
	\end{equation*}
	
	We will also use the Landau symbol $ \Oc $, i.e. ``Big-O'', in the usual way and the symbol $ \sim $ to denote asymptotic equality.
	Our result is now the following theorem:
	
	\begin{mythm}
		\label{thm:genasymp}
		Let $ p_1, \ldots, p_k, q_1, \ldots, q_l $ be non-constant complex polynomials and $ a_1, \ldots, a_k, b_1, \ldots, b_l $ be non-zero complex polynomials.
		Furthermore, assume that $ \deg p_1 > \max_{i=2,\ldots,k} \deg p_i $ and $ \deg q_1 > \max_{j=2,\ldots,l} \deg q_j $.
		Using the notation
		\begin{equation*}
			D(n,m) := \deg \setb{\sum_{i=1}^{k} a_i(x) p_i(x)^n + \sum_{j=1}^{l} b_j(x) q_j(x)^m}
		\end{equation*}
		we have
		\begin{equation*}
			A_d := \# \set{(n,m) \in \NN^2 : 0 \leq D(n,m) \leq d} \sim \frac{d^2}{\deg p_1 \cdot \deg q_1}
		\end{equation*}
		as $ d \to \infty $.
	\end{mythm}
	
	In combination with Theorems 1 und 3 in \cite{fuchs-heintze-p8}, which state that under suitable assumptions for a fixed non-zero polynomial $ f(x) $ there are only finitely many representations of $ f(x) $ of the form
	\begin{equation}
		\label{eq:reprf}
		f(x) = \sum_{i=1}^{k} a_i(x) p_i(x)^n + \sum_{j=1}^{l} b_j(x) q_j(x)^m
	\end{equation}
	and that there are only finitely many such $ f(x) $ with more than one representation of the shape \eqref{eq:reprf}, respectively, Theorem \ref{thm:genasymp} above gives also an asymptotic for the number of polynomials $ f(x) $ with $ 0 \leq \deg f \leq d $ having a representation of the form \eqref{eq:reprf}.
	
	\section{Preliminaries}
	
	The proof of our theorem given in the next section will use height functions in function fields. Hence, let us define the height of an element $ f \in F^* $ by
	\begin{equation*}
		\Hc(f) := - \sum_{\nu} \min \setb{0, \nu(f)} = \sum_{\nu} \max \setb{0, \nu(f)}
	\end{equation*}
	where the sum is taken over all valuations (up to equivalence) on the function field $ F / \CC $. Additionally we define $ \Hc(0) = \infty $.
	This height function satisfies some basic properties, listed in the lemma below which is proven in \cite{fuchs-karolus-kreso-2019}:
	
	\begin{mylemma}
		\label{lemma:heightproperties}
		Denote as above by $ \Hc $ the height on $ F/\CC $. Then for $ f,g \in F^* $ the following properties hold:
		\begin{enumerate}[a)]
			\item $ \Hc(f) \geq 0 $ and $ \Hc(f) = \Hc(1/f) $,
			\item $ \Hc(f) - \Hc(g) \leq \Hc(f+g) \leq \Hc(f) + \Hc(g) $,
			\item $ \Hc(f) - \Hc(g) \leq \Hc(fg) \leq \Hc(f) + \Hc(g) $,
			\item $ \Hc(f^n) = \abs{n} \cdot \Hc(f) $,
			\item $ \Hc(f) = 0 \iff f \in \CC^* $,
			\item $ \Hc(A(f)) = \deg A \cdot \Hc(f) $ for any $ A \in \CC[T] \setminus \set{0} $.
		\end{enumerate}
	\end{mylemma}
	
	Moreover, the following theorem due to Brownawell and Masser is an important ingredient for the proof section. It is an immediate consequence of Theorem B in \cite{brownawell-masser-1986}:
	
	\begin{mythm}[Brownawell-Masser]
		\label{thm:brownawellmasser}
		Let $ F/\CC $ be a function field in one variable of genus $ \gfr $. Moreover, for a finite set $ S $ of valuations, let $ u_1,\ldots,u_k $ be $ S $-units and
		\begin{equation*}
			1 + u_1 + \cdots + u_k = 0,
		\end{equation*}
		where no proper subsum of the left hand side vanishes. Then we have
		\begin{equation*}
			\max_{i=1,\ldots,k} \Hc(u_i) \leq \bino{k}{2} \left( \abs{S} + \max \setb{0, 2\gfr-2} \right).
		\end{equation*}
	\end{mythm}
	
	\section{Proof}
	
	We are now ready to prove our theorem about the asymptotic number of solutions to the Pillai-type equation.
	
	\begin{proof}[Proof of Theorem \ref{thm:genasymp}]
		By the dominant root condition there exist positive integers $ N $ and $ M $ such that for $ n \geq N $ we have $ \deg(a_1(x) p_1(x)^n) > \deg(a_i(x) p_i(x)^n) $ for $ i=2,\ldots,k $ and for $ m \geq M $ we have $ \deg(b_1(x) q_1(x)^m) > \deg(b_j(x) q_j(x)^m) $ for $ j=2,\ldots,l $.
		
		Since we aim for proving an asymptotic result for $ d \to \infty $, we may assume that $ d $ is large enough such that the following four inequalities are valid:
		\begin{itemize}
			\item $ \forall n < N : \deg \left( \sum_{i=1}^{k} a_i(x) p_i(x)^n \right) < d $;
			\item $ \forall m < M : \deg \left( \sum_{j=1}^{l} b_j(x) q_j(x)^m \right) < d $;
			\item $ \frac{d - \deg a_1}{\deg p_1} > N+2 $;
			\item $ \frac{d - \deg b_1}{\deg q_1} > M+2 $.
		\end{itemize}
		
		We start by proving a lower bound for $ A_d $.
		For $ (n,m) \in \NN^2 $ with $ N \leq n \leq \frac{d - \deg a_1}{\deg p_1} $ and $ M \leq m \leq \frac{d - \deg b_1}{\deg q_1} $ we clearly have $ D(n,m) \leq d $.
		Moreover, for each such $ n $ there is at most one $ m $ with $ D(n,m) < 0 $.
		Thus we get
		\begin{align}
			A_d &\geq \left( \frac{d - \deg a_1}{\deg p_1} - 1 - N \right) \left( \frac{d - \deg b_1}{\deg q_1} - 1 - M \right) - \left( \frac{d - \deg a_1}{\deg p_1} - (N-1) \right) \nonumber \\
			\label{eq:lowerbound}
			&= \frac{d^2}{\deg p_1 \cdot \deg q_1} + \Oc(d).
		\end{align}
		
		Now we need an upper bound for $ A_d $.
		Assuming $ D(n,m) \leq d $, we distinguish four cases.
		In the first case let $ n \leq \frac{d - \deg a_1}{\deg p_1} $ and $ m \leq \frac{d - \deg b_1}{\deg q_1} $. Then there are at most
		\begin{equation*}
			\left( \frac{d - \deg a_1}{\deg p_1} \right) \left( \frac{d - \deg b_1}{\deg q_1} \right)
		\end{equation*}
		such pairs.
		
		In the second case we assume $ n \leq \frac{d - \deg a_1}{\deg p_1} $ and $ m > \frac{d - \deg b_1}{\deg q_1} $.
		Then $ b_1(x) q_1(x)^m $ is the term with largest degree and we get the contradiction $ D(n,m) = \deg b_1 + m \deg q_1 > d $.
		Analogously, the third case $ n > \frac{d - \deg a_1}{\deg p_1} $ and $ m \leq \frac{d - \deg b_1}{\deg q_1} $ ends up in a contradiction.
		
		Lastly, we consider the case $ n > \frac{d - \deg a_1}{\deg p_1} $ and $ m > \frac{d - \deg b_1}{\deg q_1} $.
		Here $ D(n,m) \leq d $ can only be possible if $ \deg a_1 + n \deg p_1 = \deg b_1 + m \deg q_1 $.
		Hence $ n $ and $ m $ uniquely determine each other.
		Writing
		\begin{equation}
			\label{eq:general1}
			\sum_{i=1}^{k} a_i(x) p_i(x)^n + \sum_{j=1}^{l} b_j(x) q_j(x)^m = f(x)
		\end{equation}
		with $ 0 \leq \deg f \leq d $ we aim for applying Theorem \ref{thm:brownawellmasser}.
		Choosing a finite set $ S $ of discrete valuations such that all $ a_i, p_i, b_j, q_j $ as well as $ f $ are $ S $-units is possible with
		\begin{equation*}
			|S| \leq 1 + d + \sum_{i=1}^{k} \deg a_i + \sum_{i=1}^{k} \deg p_i + \sum_{j=1}^{l} \deg b_j + \sum_{j=1}^{l} \deg q_j.
		\end{equation*}
		Let us define
		\begin{equation*}
			C_{BM} := \bino{k+l}{2} \left( 1 + d + \sum_{i=1}^{k} \deg a_i + \sum_{i=1}^{k} \deg p_i + \sum_{j=1}^{l} \deg b_j + \sum_{j=1}^{l} \deg q_j \right)
		\end{equation*}
		and rewrite equation \eqref{eq:general1} as
		\begin{equation*}
			1 - \sum_{i=1}^{k} \frac{a_i}{f} p_i^n - \sum_{j=1}^{l} \frac{b_j}{f} q_j^m = 0.
		\end{equation*}
		Now we take a closer look at a minimal vanishing subsum containing the summand $ 1 $.
		There must be at least one further summand in this subsum.
		Assume first that this summand has the form $ \frac{b_{j_0}}{f} q_{j_0}^m $ for some $ j_0 $.
		Then, by Theorem \ref{thm:brownawellmasser}, we get
		\begin{equation*}
			\Hc \left( \frac{b_{j_0}}{f} q_{j_0}^m \right) \leq C_{BM}
		\end{equation*}
		and by some standard calculations using properties of the height function from Lemma \ref{lemma:heightproperties} (cf. e.g. the calculations in \cite{fuchs-heintze-p8}) the bound
		\begin{equation*}
			m \leq \frac{C_{BM} + d + \deg b_{j_0}}{\deg q_{j_0}} \leq \frac{2 C_{BM}}{\min_{j=1,\ldots,l} \deg q_j}.
		\end{equation*}
		If the summand in the subsum has the form $ \frac{a_{i_0}}{f} p_{i_0}^n $ for some $ i_0 $, we analogously get
		\begin{equation*}
			n \leq \frac{2 C_{BM}}{\min_{i=1,\ldots,k} \deg p_i}.
		\end{equation*}
		Thus in both subcases we have the bound
		\begin{equation*}
			\min(n,m) \leq \frac{2 C_{BM}}{\min(\min_{i=1,\ldots,k} \deg p_i, \min_{j=1,\ldots,l} \deg q_j)} \leq 2 C_{BM}.
		\end{equation*}
		Recalling that $ n $ and $ m $ determine each other uniquely, yields that there are no more than $ 4 C_{BM} $ pairs $ (n,m) $ with $ D(n,m) \leq d $ in this case.
		
		Putting the things together from all the four analyzed cases we get the final upper bound
		\begin{align}
			A_d &\leq \left( \frac{d - \deg a_1}{\deg p_1} \right) \left( \frac{d - \deg b_1}{\deg q_1} \right) + 4 C_{BM} \nonumber \\
			\label{eq:upperbound}
			&= \frac{d^2}{\deg p_1 \cdot \deg q_1} + \Oc(d).
		\end{align}
		
		From the lower bound \eqref{eq:lowerbound} and the upper bound \eqref{eq:upperbound} the statement of the theorem follows immediately.
	\end{proof}

\end{document}